\documentclass[11pt]{amsart}   

\usepackage{verbatim}
\usepackage{amssymb,amsthm,amsmath}
\usepackage{color}
\definecolor{darkblue}{rgb}{0, 0, .4}
\definecolor{grey}{rgb}{.7, .7, .7}

\usepackage{ifpdf}
\ifpdf
  \usepackage[pdftex]{epsfig}

  \usepackage{lscape}

  \usepackage[breaklinks]{hyperref}
  \hypersetup{
            colorlinks=true,
            linkcolor=darkblue,
            anchorcolor=darkblue,
            citecolor=darkblue,
            pagecolor=darkblue,
            urlcolor=darkblue,
            pdftitle={},
            pdfauthor={}
  }
\else
  \usepackage[dvips]{epsfig}
  \usepackage{lscape}

  \newcommand{\href}[2]{#2}
  \newcommand{\url}[2]{#2}
\fi

\newtheorem{theorem}{Theorem}[section]
\newtheorem{lemma}[theorem]{Lemma}

\theoremstyle{definition}

\theoremstyle{remark}

\numberwithin{equation}{section}

\theoremstyle{theorem}
\newtheorem{corollary}[theorem]{Corollary}

\newcommand{\N}[0]{\mathbb{N}}

\usepackage{graphicx}
\input xy 
\xyoption{all}

\usepackage{tikz}
\usetikzlibrary{trees}
\usepackage{ifthen}

\newcommand{\ra}{\rightarrow}

\newcommand{\SN}{\mathfrak{S}}

\newcommand{\pf}[1]{\pi|_{[#1]}}

\newcommand{\la}[0]{\lambda}
\renewcommand{\t}[0]{\theta}

\begin{document}

\title{Opportunity costs in the game of best choice}

\begin{abstract}
The game of best choice, also known as the secretary problem, is a model for sequential decision making with many variations in the literature.  Notably, the classical setup assumes that the sequence of candidate rankings is uniformly distributed over time and that there is no expense associated with the candidate interviews.  Here, we weight each ranking permutation according to the position of the best candidate in order to model costs incurred from conducting interviews with candidates that are ultimately not hired.  We compare our weighted model with the classical (uniform) model via a limiting process.  It turns out that imposing even infinitesimal costs on the interviews results in a probability of success that is about 28\%, as opposed to $1/e \approx 37\%$ in the classical case.  
\end{abstract}

\author[M. Crews]{Madeline Crews}
\author[B. Jones]{Brant Jones}
\author[K. Myers]{Kaitlyn Myers}
\author[L. Taalman]{Laura Taalman}
\author[M. Urbanski]{Michael Urbanski}
\author[B. Wilson]{Breeann Wilson}
\address{Department of Mathematics and Statistics, MSC 1911, James Madison University, Harrisonburg, VA 22807}
\email{\href{mailto:jones3bc@jmu.edu}{\texttt{jones3bc@jmu.edu}}}
\urladdr{\url{http://educ.jmu.edu/\~jones3bc/}}


\date{\today}

\maketitle

\bigskip
\section{Introduction}

The game of best choice, or secretary problem, is a model for sequential
decision making.  In the simplest variant, an {\bf interviewer} evaluates a
pool of $N$ {\bf candidates} one by one.  After each interview, the interviewer
ranks the current candidate against all of the candidates interviewed so far,
and decides whether to {\bf accept} the current current candidate (ending the
game) or to {\bf reject} the current candidate (in which case, they cannot be
recalled later).  The goal of the game is to hire the best candidate out of
$N$.  It turns out that the optimal strategy for large $N$ is to reject an
initial set of $N/e$ candidates and hire the next candidate who is better than
all of them (or the last candidate if no subsequent candidate is better).  The
probability of hiring the best candidate out of $N$ with this strategy also
approaches $1/e$.  See \cite{gilbert--mosteller} for an introduction to these
results.  Many other variations and some history have been given in
\cite{ferguson} and \cite{freeman}. 

We model interview orderings as permutations.  The permutation $\pi$ of
$N$ is expressed in one-line notation as $[\pi_1 \pi_2 \cdots \pi_N]$
where the $\pi_i$ consist of the elements $1, 2, \ldots, N$ (so each element
appears exactly once).  In the best choice game, $\pi_i$ is the rank of the
$i$th candidate interviewed in reality, where rank $N$ is best and $1$ is
worst.  What the player sees at each step, however, are relative rankings.
For example, corresponding to the interview order $\pi = [2516374]$, the player
sees the sequence of permutations $1, 12, 231, 2314, 24153, 241536, 2516374$
and must use only this information to determine when to accept a candidate,
thereby ending the game.  

Let $\SN_N$ be the set of all permutations of size $N$.  Given some statistic
$c : \SN_N \rightarrow \N$ and a positive real number $\theta$, we define a
discrete probability distribution on $\SN_N$ via
\[ f(\pi) = \frac{ \theta^{c(\pi)} }{ \sum_{\pi \in \SN_N} \theta^{c(\pi)} }. \]
Given a sequence of $i$ distinct integers, we define its {\bf flattening} to be
the unique permutation of $\{1, 2, \ldots, i\}$ having the same relative order
as the sequence.  Given a permutation $\pi$, define the {\bf $i$th prefix
flattening}, denoted $\pf{i}$, to be the permutation obtained by flattening the
sequence $\pi_1, \pi_2, \ldots, \pi_i$.  In the {\bf weighted game of best
choice}, introduced in \cite{jones18w}, some $\pi \in \SN_N$ is chosen
randomly, with probability $f(\pi)$, and each prefix flattening $\pf{1},
\pf{2}, \ldots$ is presented sequentially to the player.  If the player stops
at value $N$, they win; otherwise, they lose.  We are interested in calculating
the win probability, under optimal play, for finite $N$ as well as in the limit
as $N \rightarrow \infty$.

In this note, we follow a suggestion by the first author to let $c(\pi)$ be the
position of the largest element in $\pi$, indexed starting from $0$; that is,
$c(\pi) = \pi^{-1}(N) - 1$.  Equivalently, this is the number of ``wasted''
interviews required before we can hire the best candidate.  Setting $\t < 1$
has the effect of imposing a multiplicative cost of $\t$ on each wasted
interview.  For example, the best candidate being hired immediately will
contribute $1 = \t^0$ (before normalization) to the win probability, whereas
each failed interview reduces the contribution of an eventually successful hire
by a factor of $\t$.  This weighted model is relevant when the interviews
themselves are costly, or if time spent interviewing detracts from the time
spent working productively such as when the position being filled is only for a
limited term or requires a substantial training investment.  Also, observe that
when $\theta = 1$, we recover the complete uniform distribution on $\SN_N$,
corresponding to the classical model.

We obtain some interesting behavior vis-\`a-vis the classical model.  The
optimal strategy is still positional, for which we reject about $0.435/(1-\t)$
initial candidates and select the next best candidate.  As $N \ra \infty$ and
$\t \ra 1$, however, this strategy succeeds about $28\%$ of the time even
though we have a $1/e \approx 37\%$ success rate {\it at } $\t = 1$.  That is,
the asymptotically optimal strategy does not vary continuously with the
parameter $\t$ which seems to limit the durability of any ``policy advice''
derived from the classical model (such as e.g. \cite{golden}).  We found a
similar discontinuity in the {\em optimal strategy} for the Mallows model in
\cite{jones18w}, although the {\em success probability} there still approached
$1/e$.  In the present model, both the strategy and probability of success are
discontinuous.  Evidently, there is a ``price'' of about $8.6\%$ in the
asymptotic success rate for imposing any wasted interview penalty, no matter
how small.

Although there is an established ``full-information'' version of the game in
which the player observes values from a given distribution, it seems that only
a few papers have considered nonuniform rank distributions for the secretary
problem.  Pfeifer \cite{pfeifer} considers the case where interview ranks are
independent but have cumulative distribution functions containing parameters
determined by the interview positions.  The paper \cite{reeves--flack}
considers an explicit continuous probability distribution that allows for
dependencies between nearby arrival ranks via a single parameter.  Inspired by
approximation theory, the paper \cite{kleinberg--etal} studies some general
properties of non-uniform rank distributions in the secretary problem.  Our
work also fits into a recent stream of asymptotic results for random
permutations by researchers in algebraic combinatorics such as
\cite{miner-pak,elizalde,bouvel}.

\bigskip
\section{The Model}

The {\bf left-to-right maxima} in a permutation $\pi$ consist of elements
$\pi_j$ that are larger in value than every element $\pi_i$ to the left (i.e.
for $i<j$).  In the game of best choice, it is never optimal to select a
candidate that is not a left-to-right maximum.  A {\bf positional} strategy for
the game of best choice is one in which the interviewer transitions from
rejection to hiring based only on the position of the interview (as opposed to
adjusting the transition based on the prefix flattenings that are
encountered).  More precisely, the interviewer may play the {\bf
$r$-positional} strategy on a permutation $\pi$ by rejecting candidates $\pi_1,
\pi_2, \ldots, \pi_r$ and then accepting the next left-to-right maximum
thereafter.  We say that a particular interview rank order is {\bf
$r$-winnable} if transitioning from rejection to hiring after the $r$th
interview captures the best candidate.  For example, $574239618$ is
$r$-winnable for $r = 2, 3, 4$, and $5$.  It is straightforward to verify that
a permutation $\pi$ is $r$-winnable precisely when position $r$ lies between
the last two left-to-right maxima in $\pi$.  

It follows from the results in \cite[Section 3]{jones18w} that the optimal
strategy in our game of best choice is positional \footnote{Briefly, our
statistic $c(\pi) = \pi^{-1}(N)-1$ is essentially prefix equivariant
\cite[Definition 3.2]{jones18w} in the sense that $c(\pi) - c(\sigma_q \cdot
\pi) = 0 = c(12 \cdots k) - c(q)$ for all eligible prefixes $q$.  This is
enough to obtain the results in \cite[Theorem 3.4]{jones18w} and subsequently
\cite[Theorem 3.7]{jones18w}.}, and we let
\[ W_N(r)=\sum_{\text{$r$-winnable $\pi \in \SN_N$}} \t^{\pi^{-1}(N)-1}. \] 

\begin{theorem}\label{t:rec}
We have the recurrence 
\[ W_{N}(r) = (N-1) W_{N-1}(r) + r (N-2)! \t^{N-1} \]
with initial conditions $W_1(0) = 1$ and $W_1(r) = 0$ for all $r \geq 1$.
\end{theorem}
\begin{proof}
There are two cases for the $r$-winnable permutations $\pi$ of $N$.  If $N$
does not lie in the last position, then we may view the initial segment of
$\pi$ uniquely as an $r$-winnable permutation of $N-1$ by flattening.  Since
there are $N-1$ possible values for the last position, this case contributes $(N-1)
W_{N-1}(r)$ to $W_N(r)$.  If $N$ lies in the last position, then $\pi$ will be
winnable if and only if $N-1$ lies in one of the first $r$ positions of $\pi$.
For each of these choices, we may permute the remaining entries in $(N-2)!$
ways, so these contribute $r (N-2)! \t^{N-1}$ all together. 
\end{proof}

\begin{corollary}
We have 
\[ W_N(r) = \begin{cases} 
(N-1)! & \text{ if $r=0$ } \\
(N-1)! \ r \sum\limits_{i=r}^{N-1} \frac{\t^i}{i} & \text{ if $1 \leq r \leq N-1$ } \\ 
\end{cases} \]
\end{corollary}
\begin{proof}
This follows from Theorem~\ref{t:rec} by induction.
\end{proof}

\begin{theorem}\label{t:succN}
Fix some positive $\t \neq 1$.  The probability of winning the game of best choice using the strategy that rejects $r$ initial candidates is
\[ P_r(N, \t) = \frac{r (1-\t) \sum_{i=r}^{N-1} \frac{\t^i}{i} }{1-\t^N} \]
if $r > 0$ and is $\frac{1-\t}{1-\t^N}$ if $r = 0$.
\end{theorem}
\begin{proof}
By definition, the probability of winning is
\[ \biggm( \sum\limits_{r\text{-winnable } \pi \in \SN_N} \theta^{c(\pi)} \biggm)\Big/\biggm( \sum\limits_{\pi \in \SN_N}
\theta^{c(\pi)} \biggm)  = \frac{W_N(r)}{(N-1)!(1+\t+\t^2+\cdots+\t^{N-1})}. \]
The result then follows from the previous corollary.
\end{proof}

\bigskip
\section{Results}

Suppose now that $\t < 1$, and take the limit as $N \ra \infty$.
Then the probability of success for the strategy that initially rejects $r$
candidates becomes
\[ P_r(\t) = r (1-\t) \sum_{i=r}^{\infty} \frac{\t^i}{i}. \]
We obtain a curve for each nonnegative value of $r$ (interpreting $P_0$ as
$1-\t$), the first several of which we have plotted in Figure~\ref{f:curves}.
For each value of $\t$, one of the curves is maximal, yielding the optimal strategy and probability of success.  For example, 
\[ \begin{tabular}{|c|cccc|}
	\hline
	$r = $ & 0 & 1 & 2 & 3 \\
	\text{ is optimal for $\t \in$} & (0, 0.6321] & [0.6321, 0.7968] &  [0.7968, 0.8609] &  [0.8609, 0.8945] \\
	\hline
\end{tabular} \]
\begin{figure}[t]
\begin{center}
\includegraphics[scale=0.48]{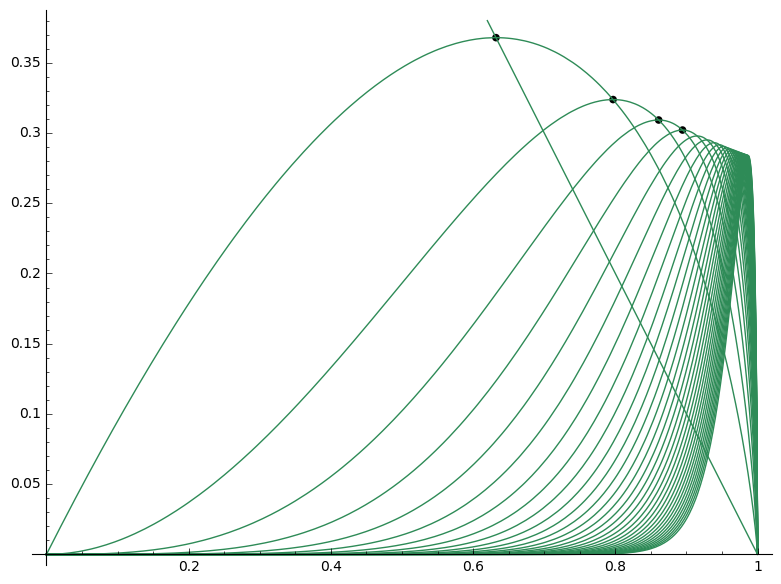}
\end{center}
\caption{The first few asymptotic curves $P_0, P_1, \ldots$}\label{f:curves}
\end{figure}

\begin{lemma}\label{l:inter}
For each $i$, the intersection of $P_{i-1}$ and $P_i$ coincides with the maximum value of $P_i$.  
\end{lemma}
\begin{proof}
To see this, the derivative of $P_r$ with respect to $\t$ is 
		\[ r(1-\t) \sum_{i=r}^\infty \t^{i-1} - r \sum_{i=r}^\infty \frac{\t^i}{i} =  r(1-\t) \frac{\t^{r-1}}{1-\t} - r \sum_{i=r}^\infty \frac{\t^i}{i}   =  r \left(\t^{r-1} - \sum_{i=r}^\infty \frac{\t^i}{i}\right). \]
		whereas the successive differences $P_{r-1}-P_r$ are
		\[ (1-\t)\left( ( (r-1)-r) \sum_{i=r}^\infty \frac{\t^i}{i} \right) + (r-1)(1-\t)\frac{\t^{r-1}}{r-1} = (1-\t)\left( \t^{r-1} - \sum_{i=r}^\infty \frac{\t^i}{i} \right).\]
Hence, 
		\[ P_{r-1}-P_r = \frac{1-\t}{r} \frac{d P_r}{d\t} \]
so the successive differences and derivatives have the same zeros.
\end{proof}

The first intersection occurs at $\t = 1-1/e$ with value $P = 1/e$.  Subsequent intersections can be estimated numerically but have no elementary closed form:
\[ \begin{array}{|c|c|c|c|}
	\hline
	r & \frac{d}{d\t}P_r & \text{solution for $\frac{d}{d\t}P_r=0$} & \text{value of $P_r$} \\
	\hline
	1 & \ln(-\t + 1) + 1 & \t = 1-1/e \approx 0.63212 & 1/e \approx 0.36788 \\
	2 & 4\t + 2\ln(-\t + 1) & \t \approx 0.796812 & 0.323805 \\
	3 & 9/2\t^2 + 3\t + 3\ln(-\t + 1) & \t \approx 0.860917 & 0.309256 \\
	4 & 16/3\t^3 + 2\t^2 + 4\t + 4\ln(-\t + 1) & \t \approx 0.894457 & 0.302113 \\
	5 & 25/4\t^4 + 5/3\t^3 + 5/2\t^2 + 5\t + 5\ln(-\t + 1) & \t \approx 0.915009 & 0.297883 \\
	\hline
\end{array} \]

Thus, the optimal strategy and probability of success is determined by
$P(\theta) = \max\limits_{r \geq 0} P_r(\theta)$, the maximum $P_r$ function in
the regime determined by $\theta$.  By Lemma~\ref{l:inter}, $P(\theta)$ is
monotonically decreasing and bounded below.  Hence, there is a limiting value
as $\t \ra 1$.  However, the limit is clearly bounded away from $1/e$, which is
the value {\em at } $\t = 1$ according to the classical analysis.  Our goal in
this section is to determine $\lim_{\theta \rightarrow 1} P(\theta)$ more precisely.

\begin{figure}[t]
\begin{center}
\includegraphics[scale=0.4]{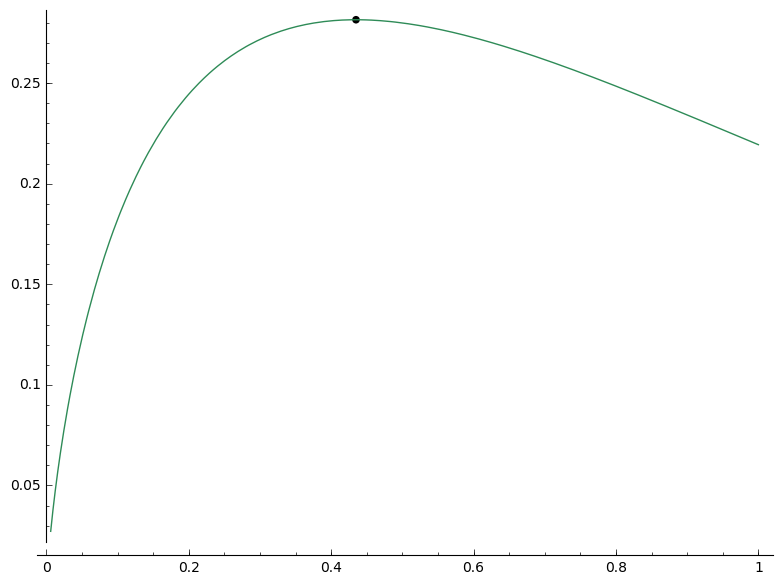}
\end{center}
\caption{The maximum $(\alpha, \beta)$ for $x E_1(x)$}\label{f:xe1x}
\end{figure}

Recall the exponential integral
\[ E_1(x) = \int_{x}^\infty \frac{e^{-t}}{t} \ dt \]
which we view as a function of a positive real variable $x$ (see e.g. \cite{NIST}).
This is a standard special function implemented in many mathematical software
systems.

For our main result, we consider the maximum value attained by the related
function $F(x) = x E_1(x)$ on $(0, \infty)$; see Figure~\ref{f:xe1x} for a
plot.  Although there is no elementary form for this maximum, it occurs where
$E_1(x) = e^{-x}$ so can be estimated numerically to arbitrary precision.  Let
$\alpha$ and $\beta$ be defined by $F'(\alpha) = 0$ and $F(\alpha) = \beta$.
Then, $\alpha \approx 0.43481821500399293$ and $\beta \approx
0.28149362995691674$.


We are now in a position to give our main result.

\begin{theorem}
As $\t$ approaches $1$ from the left, the optimal strategy in our asymptotic weighted game of best choice approaches a positional strategy that rejects $\frac{\alpha}{1-\t}$ initial candidates and selects the next candidate better than all of them.  This strategy has a success probability of $\beta$.
\end{theorem}
\begin{proof}
We would like to optimize
$P_r(\t) = r (1-\t) \sum_{i=r}^{\infty} \frac{\t^i}{i}$
for large $r$ and $\t$ chosen appropriately close to $1$.
We estimate the series by viewing it as a left or right sum for the corresponding integrals:
\[ \int_{t=r}^\infty \frac{\t^t}{t} \ dt < \sum_{i=r}^{\infty} \frac{\t^i}{i} < \int_{t=r}^\infty \frac{\t^{t-1}}{t-1} \ dt = \int_{t=r-1}^\infty \frac{\t^{t}}{t} \ dt. \]
Hence, we may approximate $P_r(\t)$ by
$\widetilde{P}_r(\t) = r (1-\t) \int_{t=r}^{\infty} \frac{\t^t}{t} \ dt$
with error less than 
\[ r(1-\t) \frac{\t^{r-1}}{r-1} < 4 (1-\t) \t^r \]
since the integrand is decreasing, $r-1 > r/2$, and $\t > 1/2$.

Next, we change variables from $r$ to $c = (1-\t)r$, and from $t$ to $u = (1-\t)t$ in the integral.
We obtain $du = (1-\t)\ dt$ so
\[ \widetilde{P}_c(\t) = c \int_{u=c}^{\infty} \frac{\left(\t^{1/(1-\t)}\right)^u}{u} \ du. \]
and our error estimate for $|P - \widetilde{P}|$ becomes $4\t^{c/(1-\t)} (1-\t)$.

Now, we are in a position to take the limit as $\t \rightarrow 1$, using $\lim_{\t \ra 1} \t^{1/(1-\t)} = 1/e$.  This forces $P \rightarrow \widetilde{P}$ by our error estimate, and 
\[ \widetilde{P}_c \rightarrow c \int_{u=c}^{\infty} \frac{e^{-u}}{u} \ du = c E_1(c). \]
Optimizing this function for $c \in (0, \infty)$ then determines the asymptotically optimal positional strategy (where we reject $r = \frac{c}{1-\t}$ initial candidates) and probability of success.
\end{proof}

We can also solve the model when $\t > 1$.  One interpretation here is that
there is some ``trend'' in the candidate pool (e.g. due to changes in general
economic conditions such as unemployment or interest rates) that is amplifying
the probability of seeing the best candidate later.  Once again, we find that
including even an infinitesimal trend completely changes the optimal asymptotic
strategy.

\begin{theorem}
If $\t > 1$, the probability of success for the strategy that initially rejects $r = \frac{N}{\la}$ candidates approaches $1/\la$, as $N \ra \infty$.  Hence, the asymptotic model does not depend on $\t$.
\end{theorem}
\begin{proof}
Recall Theorem~\ref{t:succN}; we claim
\[ \lim_{N \ra \infty} P_{N/\la}(N, \t) = \lim_{N \ra \infty} \frac{N (1-\t)}{\la(1-\t^N)} \sum_{i=N/\la}^{N-1} \frac{\t^i}{i} = \frac{1}{\la}. \]

To see why, consider the ``almost telescoping'' sum
\[ (1-\t) \sum_{i=N/\la}^{N-1} \frac{\t^i}{i} = \frac{1}{N/\la} \t^{N/\la} - \frac{1}{N-1} \t^{N} - \sum_{i=(N/\la) + 1}^{N-1} \frac{1}{i(i-1)} \t^i. \]
If we divide by $-\t^N$ and take the limit $N \ra \infty$ term by term with $\t > 1$, we find that only the leading (i.e. middle) term survives.  
Hence, our limit is $\lim_{N \rightarrow \infty} \frac{N}{\la(1-\t^N)} \left( -\frac{1}{N-1} \t^N \right) = \frac{1}{\la}$.
\end{proof}

Thus, for large $N$ we find that it is optimal to choose the last candidate (and win almost all of the time!), obtaining another discontinuity with the $\t < 1$ and classical models.

\bigskip
\section*{Addendum}

{\footnotesize Although we were unaware of it until after our work was accepted for publication, the paper \cite{rasm} solves a very similar problem to the one we are considering.  More specifically, we compute in Theorem~\ref{t:succN} the probability of winning the game of best choice under a non-uniform distribution whereas Rasmussen--Pliska compute in their Equation (2.6) the expected value of a random variable representing the non-uniform payoff for the game played on a uniform distribution.  For any particular $N$ and $\theta$, these problems are dual to each other in the sense that their corresponding formulas are off by the multiplicative constant $\frac{\t+\t^2+\t^3+...+\t^N}{N}$.  Since our weights form a probability distribution, we believe our model facilitates a clearer comparison with the classical secretary problem.

When $0 < \theta < 1$, Rasmussen--Pliska obtain an asymptotic estimate for the optimal strategy that agrees with ours, using different methods.  They also note that, for fixed $\theta < 1$, their expected payoff tends to $0$ as $N$ tends to infinity (because their denominator does not scale with $\theta$).  By contrast, our model has nonzero probabilities, as $N$ tends to infinity, given by the value of $P_r(\theta)$ where $r$ is optimal for the fixed $\theta$.}


\bigskip
\section*{Acknowledgements}

This material is based upon work supported by the National Science Foundation
under Grant Number 1560151.  It was initiated during the summer of 2018 at a
research experience for undergraduates (REU) site at James Madison
University mentored by the second and fourth authors.



\end{document}